\documentclass[reqno]{amsart}
\usepackage{amsthm,amsmath,amssymb}
\usepackage{stmaryrd,mathrsfs}
\usepackage{hyperref}
\usepackage{esint}
\usepackage{tikz}
\usepackage{amsmath}

\usetikzlibrary{patterns,arrows,decorations.pathreplacing}

\allowdisplaybreaks

\newcommand{\dist}[0]{\operatorname{dist}}

\swapnumbers \numberwithin{equation}{section}

\theoremstyle{plain}
\newtheorem{theorem}[equation]{Theorem}
\newtheorem{proposition}[equation]{Proposition}
\newtheorem{corollary}[equation]{Corollary}
\newtheorem{lemma}[equation]{Lemma}

\theoremstyle{definition}
\newtheorem{definition}[equation]{Definition}

\theoremstyle{remark}
\newtheorem{remark}[equation]{Remark}

\makeatletter
\@namedef{subjclassname@2010}{%
  \textup{2010} Mathematics Subject Classification}
\makeatother

\begin{document}

\title[Equality of $VJN_p$ and $CJN_p$]{The John-Nirenberg space: equality of the vanishing subspaces $VJN_p$ and $CJN_p$}

\author[R. Korte]{Riikka Korte}
\address{(R.K.) Aalto University, Department of Mathematics and Systems analysis, Espoo, Finland}
\email{riikka.korte@aalto.fi}

\author[T. Takala]{Timo Takala}
\address{(T.T.) Aalto University, Department of Mathematics and Systems analysis, Espoo, Finland}
\email{timo.i.takala@aalto.fi}

%\let\thefootnote\relax\footnotetext{Affiliations of the authors: Aalto University, Department of Mathematics and Systems analysis, Espoo, Finland}
%\let\thefootnote\relax\footnotetext{Email of the corresponding author: timo.i.takala@aalto.fi}

%\thanks{}

\date{\today}


% 42B30 $H^p$-spaces
% 42B35 Function spaces arising in harmonic analysis 
% 46E30 Spaces of measurable functions ($L^p$-spaces, Orlicz spaces, K\"othe function spaces, Lorentz spaces, rearrangement invariant spaces, ideal spaces, etc.)
% 46E35 Sobolev spaces and other spaces of "smooth" functions, embedding theorems, trace theorems

\maketitle

\begin{abstract}
%We study weak $L^p$ type integrals of $JN_p$ functions on cubes. We show that these integrals tend to zero, both when the measure of the cube tends to zero and when the measure tends to infinity. In particular this shows that $VJN_p$ and $CJN_p$ coincide. These are vanishing subspaces of $JN_p$ that are defined in a similar way as $VMO$ and $CMO$, which are subspaces of $BMO$.
%We show that the space $JN_{p,q}(\mathbb{R}^n)$, which uses a more general definition, is equivalent with $L^p(\mathbb{R}^n) / \mathbb{R}$, if $p = q$.

The John-Nirenberg spaces $JN_p$ are generalizations of the space of bounded mean oscillation $BMO$ with $JN_\infty=BMO$. Their vanishing subpaces $VJN_p$ and $CJN_p$ are defined in similar ways as $VMO$ and $CMO$, which are subspaces of $BMO$. 
%In this paper, we study $JN_p$-spaces and their vanishing subspaces $VJN_p$ and $CJN_p$, which are defined in a similar way as $VMO$ and $CMO$. 
As our main result, we prove that $VJN_p$ and $CJN_p$ coincide by showing that certain Morrey type integrals of $JN_p$ functions tend to zero for small and large cubes. We also show that $JN_{p,q}(\mathbb{R}^n) =L^p(\mathbb{R}^n) / \mathbb{R}$, if $p = q$.

\end{abstract}

\section{Introduction}

In 1961 John and Nirenberg studied the well known space of bounded mean oscillation $BMO$ and proved the profound John-Nirenberg inequality for $BMO$ functions \cite{johnnirenberg}. The space $BMO$ plays a vital role in harmonic analysis and it has been studied very extensively. In the same article they also defined a generalization of $BMO$, which is now known as the John-Nirenberg space, or $JN_p$, with a parameter $1 < p < \infty$, see Definition \ref{jnpmaaritelma} below. In addition they proved the John-Nirenberg inequality for $JN_p$ functions, see Theorem \ref{johnnirenberglemma} below. From this theorem it follows that $JN_p(Q_0) \subset L^{p,\infty}(Q_0)$, where $Q_0 \subset \mathbb{R}^n$ is a bounded cube. It is also easy to see that $L^p(Q_0) \subset JN_p(Q_0)$. Both of these inclusions are strict, however this is far from trivial. An example of a function in $JN_p \setminus L^p$ was discovered in 2018 \cite{dafnihytonenkorteyue}. Thus the space $JN_p$ is a nontrivial space between $L^p$ and $L^{p,\infty}$. However there are still many unanswered questions related to the study of John-Nirenberg spaces.

Various John–Nirenberg type spaces have attracted attention in recent years, including the dyadic John-Nirenberg space \cite{kinnunenmyyrylainen}, the congruent John-Nirenberg space \cite{jiataoyangyuanzhang, taoyangyuancongruent}, the John-Nirenberg-Campanato space \cite{taoyangyuancampanato, sunxieyang} and the sparse John-Nirenberg space \cite{dominguezmilman}.
The John-Nirenberg space can also be defined with medians instead of using integral averages \cite{myyrylainen}.
Hurri-Syrj\"anen et al. established a local-to-global result for the space $JN_p(\Omega)$, where $\Omega \subset \mathbb{R}^n$ is an open set \cite{hurrisyrjanenmarolavahakangas} and Marola and Saari found similar results for $JN_p$ in the setting of metric measure spaces \cite{marolasaari}.

The spaces $VMO$ and $CMO$ are well known vanishing subspaces of $BMO$. They were introduced by Sarason \cite{sarason} and Neri \cite{neri} respectively. Recently there has been research on the $JN_p$ counterparts of these spaces, which are denoted by $VJN_p$ and $CJN_p$ \cite{taoyangyuanvjnp}. It follows directly from the definitions that $L^p \subseteq CJN_p \subseteq VJN_p \subseteq JN_p$. Moreover, examples in \cite{takala} demonstrate that  $L^p \neq CJN_p$ and $VJN_p \neq JN_p$. However it has been an open question whether the set $VJN_p\setminus CJN_p$ is nonempty, see \cite{takala,taoyangyuanvjnp}. As our main result we show that $VJN_p$ and $CJN_p$ coincide.

Our method is to study Morrey or weak $L^p$ type integrals 
\begin{equation}\label{eqn:morrey}
|Q|^{\frac{1}{p}-1} \int_Q |f|,
\end{equation}
where $Q$ is a cube.
%Our main results have to do with these integrals, which we call weak $L^p$ type integrals based on the expression of the weak $L^p$ norm in Definition \ref{weaklpdefinition}.
We prove that if $f \in JN_p$, then these integrals tend to zero both when $|Q| \to 0$ and when $|Q| \to \infty$. See Theorems \ref{vanishingbigcubes} and \ref{vanishingsmallcubes} below for precise statements of these results. %Theorem \ref{vanishingsmallcubes} answers a question that was posed in \cite{takala}. 
Note that $L^p$ functions have this property, but weak $L^p$ functions do not.
%The results are relatively easy to show for $L^p$ functions - see Remarks \ref{firstremark} and \ref{secondremark} below - but they are not true for weak $L^p$ functions. This is why the weak $L^p$ type integrals require independent study for $JN_p$ functions.
From Theorem \ref{vanishingbigcubes} it follows easily that $CJN_p = VJN_p$, see  Corollary \ref{cjnpvjnptulos}.

%There has been also some interest in studying more general versions of John-Nirenberg norms. 
In Section \ref{preliminaries} we briefly study the more general version of the John-Nirenberg type spaces $JN_{p,q}(X)$, where the $L^1$-norm of the oscillation term is replaced with the $L^q$-norm where $q \geq 1$. This generalization has been studied in \cite{dafnihytonenkorteyue, taoyangyuansurvey, taoyangyuancongruent}. It has turned out that in case $X$ is a bounded cube, the $JN_{p,q}$ norm is equivalent with the $JN_p$ norm (for $q<p$) or $L^q$ norm (for $q \geq p$). In case $X = \mathbb{R}^n$, the $JN_{p,q}$ norm is equivalent with the $JN_p$ norm (for $q<p$), and if $q>p$, the space contains only functions that are constant almost everywhere. We complete this picture by showing that in the borderline case $p = q$ and $X = \mathbb{R}^n$, this space is equivalent with the space $L^p(\mathbb{R}^n) / \mathbb{R}$ i.e. the space of functions $f$ for which there is a constant $b$ such that $f - b \in L^p(\mathbb{R}^n)$. The result answers a question raised in \cite[Remark 2.9]{taoyangyuancongruent}.

\section{Preliminaries}\label{preliminaries}
Throughout this paper by a cube we mean an open cube with edges parallel to the coordinate axes. We let $X \subseteq \mathbb{R}^n$ be either a bounded cube or the entire space $\mathbb{R}^n$.
%For $\Omega \subset \mathbb{R}^n$ we denote by $\chi_{\Omega}$ the characteristic function of $\Omega$.
If $Q$ is a cube, we denote by $l(Q)$ its side length. For any $r > 0$, we denote by $rQ$ the cube with the same center as $Q$ but with side length $r \cdot l(Q)$. For any measurable set $E \subset \mathbb{R}^n$, such that $0 < |E| < \infty$, we denote the integral average of a function $f$ over $E$ by
\begin{equation*}
f_E
:= \fint_E f
:= \frac{1}{|E|} \int_E f.
\end{equation*}

\begin{definition}[Weak $L^p$-spaces]
\label{weaklpdefinition}
Let $1 \leq p < \infty$. For a measurable function $f$ we define
\begin{equation*}
\| f \|_{L^{p,\infty}(X)}
:= \sup_{t > 0} t | \{ x \in X : |f(x)| > t \} |^{1/p}.
\end{equation*}
We say that $f$ is a weak $L^p$ function, or $f \in L^{p,\infty}(X)$, if $\| f \|_{L^{p,\infty}(X)}$ is finite.
We define
\begin{equation*}
\| f \|_{L^{p,w} (X)}
:= \sup_{\substack{E \subseteq X \\ 0 < |E| < \infty }} |E|^{1/p} \fint_E |f(x)| dx,
\end{equation*}
where $E$ is any measurable set. We say that $f \in L^{p,w}(X)$, if $\| f \|_{L^{p,w}(X)}$ is finite.
\end{definition}

\begin{remark}
\label{weaklpremark}
The expression $\| \cdot \|_{L^{p,\infty}(X)}$ is not a norm, since the triangle inequality fails to hold. However $\| \cdot \|_{L^{p,w}(X)}$ does define a norm.
Additionally, if $p > 1$, $\| f \|_{L^{p,\infty}(X)}$ and $\| f \|_{L^{p,w}(X)}$ are comparable and therefore $L^{p,w}(X) = L^{p,\infty}(X)$, see Chapter 2.8.3 in \cite{duoandikoetxea}.
%For the proof of the equivalence, we refer to Exercises 1.1.11. and 1.1.12. in \cite{grafakos}.
%If $p=1$, then clearly $\| f \|_{L^{p,w}(X)} = \| f \|_{L^p(X)}$.
\end{remark}

\begin{definition}[$JN_p$]
\label{jnpmaaritelma}
Let $1 < p < \infty$. A function $f$ is in $JN_p(X)$ if $f \in L_{loc}^1(X)$ and there is a constant $K < \infty$ such that
\begin{equation*}
\sum_{i=1}^{\infty} |Q_i| \left( \fint_{Q_i} | f - f_{Q_i} | \right)^p \leq K^p
\end{equation*}
for all countable collections of pairwise disjoint cubes $(Q_i)_{i=1}^{\infty}$ in $X$.
We denote the smallest such number $K$ by $\|f\|_{JN_p(X)}$.
\end{definition}

The space $JN_p$ is related to $BMO$ in the sense that the $BMO$ norm of a function is the limit of the function's $JN_p$ norm when $p$ tends to infinity.
It is easy to see that
\begin{equation}
\label{itseisarvoepayhtalo}
\| |f| \|_{JN_p(X)}
\leq 2 \| f \|_{JN_p(X)}.
\end{equation}
Likewise it is clear that $L^p(X) \subset JN_p(X)$, as we get from H\"older's inequality that
$\| f \|_{JN_p(X)} \leq 2 \| f \|_{L^p(X)}$.
If $X$ is a bounded cube, then $JN_p(X) \subset L^{p, \infty} (X)$.
This is known as the John-Nirenberg inequality for $JN_p$ functions.

\begin{theorem}[John-Nirenberg inequality for $JN_p$]
\label{johnnirenberglemma}
Let $1 < p < \infty$, $Q_0 \subset \mathbb{R}^n$ a bounded cube and $f \in JN_p(Q_0)$. Then $f \in L^{p,\infty}(Q_0)$ and
\begin{equation*}
\| f - f_{Q_0} \|_{L^{p,\infty}(Q_0)}
\leq c \| f \|_{JN_p(Q_0)}
\end{equation*}
with some constant $c = c(n,p)$.
\end{theorem}

The proof can be found in \cite{aaltoberkovitskansanenyue, johnnirenberg} for example. In \cite{takala} this result was extended to the space $JN_p(\mathbb{R}^n)$.
Also a more general John-Nirenberg space $JN_{p,q}$ has been studied for example in \cite{dafnihytonenkorteyue,taoyangyuansurvey,taoyangyuancongruent}.
%One can also define the space $JN_p$ in a more general form as $JN_{p,q}$.

\begin{definition}[$JN_{p,q}$]
Let $1 \leq p < \infty$ and $1 \leq q < \infty$. A function $f$ is in $JN_{p,q}(X)$ if $f \in L_{loc}^1(X)$ and there is a constant $K < \infty$ such that
\begin{equation*}
\sum_{i=1}^{\infty} |Q_i| \left( \fint_{Q_i} | f - f_{Q_i} |^q \right)^{p/q} \leq K^p
\end{equation*}
for all countable collections of pairwise disjoint cubes $(Q_i)_{i=1}^{\infty}$ in $X$. We denote the smallest such number $K$ by $\|f\|_{JN_{p,q}(X)}$.
\end{definition}

%However this leads to nothing new, 
It was shown in \cite[Proposition 5.1]{dafnihytonenkorteyue} that if $X$ is a bounded cube, then $JN_{p,q}(X) = JN_p(X)$, if $1 \leq q < p$ and $JN_{p,q}(X) = L^q(X)$ if $p \leq q < \infty$.
The same proof also shows us that $JN_{p,q}(\mathbb{R}^n) = JN_p(\mathbb{R}^n)$, if $1 \leq q < p$.
It was shown in \cite[Corollary 2.8]{taoyangyuancongruent} that the space $JN_{p,q}(\mathbb{R}^n)$ contains only functions that are constant almost everywhere, if $p < q < \infty$. However it was stated in \cite[Remark 2.9]{taoyangyuancongruent} that the situation is unclear if $q = p$. We complete the picture by showing that $JN_{p,p}(\mathbb R^n)$ is equal to $L^p(\mathbb R^n)$ up to a constant. This also answers \cite[Question 15]{taoyangyuansurvey}.

\begin{proposition}
Let $1 \leq p < \infty$. Then $JN_{p,p}(\mathbb{R}^n) = L^p(\mathbb{R}^n) / \mathbb{R}$ and there is a constant $c = c(p)$ such that for any function $f \in L_{loc}^1(\mathbb{R}^n)$ we have
\begin{equation*}
\frac{1}{c} \| f \|_{JN_{p,p}(\mathbb{R}^n)}
\leq \inf_{b \in \mathbb{R}} \| f-b \|_{L^p(\mathbb{R}^n)}
\leq c \| f \|_{JN_{p,p}(\mathbb{R}^n)}.
\end{equation*}
\end{proposition}
\begin{proof}
First assume that $f \in L^p / \mathbb{R}$, that is there is a constant $b$ such that $f - b \in L^p$. Then for any set of pairwise disjoint cubes $Q_i$
\begin{equation*}
\sum_{i=1}^{\infty} \int_{Q_i} | f - f_{Q_i} |^p
= \sum_{i=1}^{\infty} \int_{Q_i} | (f-b) - (f-b)_{Q_i} |^p
\leq \sum_{i=1}^{\infty} 2^{p} \int_{Q_i} |f-b|^p
\leq 2^p \int_{\mathbb{R}^n} |f-b|^p
\end{equation*}
and therefore $f \in JN_{p,p}(\mathbb{R}^n)$.
\\
Now assume that $f \in JN_{p,p}(\mathbb{R}^n)$. Clearly
\begin{equation*}
\int_Q |f - f_Q|^p
\leq \| f \|_{JN_{p,p}(\mathbb{R}^n)}^p
\end{equation*}
for every cube $Q \subset \mathbb{R}^n$.
Let $(Q_k)_{k=1}^{\infty}$ be a sequence of cubes such that the center of every cube is the origin and $|Q_k| = 2^k$. Then $Q_1 \subset Q_2 \subset ...$ and $\cup_{k=1}^{\infty} Q_k = \mathbb{R}^n$. We shall prove that the sequence of integral averages $(f_{Q_k})_{k=1}^{\infty}$ is a Cauchy sequence. For any integer $i$ we have
\begin{equation*}
|f_{Q_i} - f_{Q_{i+1}}|
\leq \fint_{Q_i} |f - f_{Q_{i+1}}|
\leq 2 \fint_{Q_{i+1}} |f - f_{Q_{i+1}}|.
\end{equation*}
This means that
\begin{equation*}
\begin{split}
|f_{Q_i} - f_{Q_{i+1}}|^p
&\leq 2^p \left( \fint_{Q_{i+1}} |f - f_{Q_{i+1}}| \right)^p
\leq 2^p \fint_{Q_{i+1}} |f - f_{Q_{i+1}}|^p
\\
&\leq 2^{p-i-1} \| f \|_{JN_{p,p}(\mathbb{R}^n)}^p.
\end{split}
\end{equation*}
Then for any positive integers $m$ and $k$ we get
\begin{equation}\label{eqn:cauchyspeed}
\begin{split}
|f_{Q_m} - f_{Q_k}|
&\leq \sum_{i = \min(m,k)}^{\max(m,k) - 1} |f_{Q_{i+1}} - f_{Q_i}|
\leq \sum_{i = \min(m,k)}^{\infty} 2^{1 - \frac{1}{p}} \| f \|_{JN_{p,p}(\mathbb{R}^n)} 2^{-i/p}
\\
&= c \| f \|_{JN_{p,p}(\mathbb{R}^n)} 2^{-\min(m,k)/p},
\end{split}
\end{equation}
where the constant $c$ depends only on $p$. Therefore $(f_{Q_k})_{k=1}^{\infty}$ is a Cauchy sequence. Then by using \eqref{eqn:cauchyspeed} we get
\begin{align*}
\int_{\mathbb{R}^n} \left| f - \lim_{k \to \infty} f_{Q_k} \right|^p
&= \lim_{m \to \infty} \int_{Q_m} \left| f - \lim_{k \to \infty} f_{Q_k} \right|^p
\\
&\leq \lim_{m \to \infty} \int_{Q_m} 2^{p-1} \left( |f - f_{Q_m}|^p + \left| f_{Q_m} - \lim_{k \to \infty} f_{Q_k} \right|^p \right)
\\
&\leq \lim_{m \to \infty} 2^{p-1} \left( \| f \|_{JN_{p,p}(\mathbb{R}^n)}^p + 2^m \lim_{k \to \infty} \left| f_{Q_m} - f_{Q_k} \right|^p \right)
\\
&\leq 2^{p-1} \| f \|_{JN_{p,p}(\mathbb{R}^n)}^p \left( 1 + \lim_{m \to \infty} 2^m \lim_{k \to \infty} c 2^{-\min(m,k)} \right)
\\
&= 2^{p-1} \| f \|_{JN_{p,p}(\mathbb{R}^n)}^p \left( 1 + \lim_{m \to \infty} 2^m c 2^{-m} \right)
= c \| f \|_{JN_{p,p}(\mathbb{R}^n)}^p,
\end{align*}
where the constant $c$ depends only on $p$.
This means that $f - \lim_{k \to \infty} f_{Q_k} \in L^p$ and therefore $f \in L^p / \mathbb{R}$. This completes the proof.

\end{proof}

The spaces $VJN_p$ and $CJN_p$ were studied in \cite{taoyangyuanvjnp, takala}. These spaces are $JN_p$ counterparts of the spaces $VMO$ and $CMO$, which are subspaces of $BMO$.
The spaces $VJN_p$ and $CJN_p$ can also be defined in a bounded cube $Q_0$ instead of $\mathbb{R}^n$ as in the following definitions. However in that case it is clear that the spaces coincide \cite{taoyangyuanvjnp, takala}. In this paper we only define the spaces in $\mathbb{R}^n$ and we write $VJN_p = VJN_p(\mathbb{R}^n)$ and $CJN_p = CJN_p(\mathbb{R}^n)$ to simplify the notation.

\begin{definition}[$VJN_p$]
Let $1 < p < \infty$. Then the vanishing subspace $VJN_p$ of $JN_p$ is defined by setting
\begin{equation*}
VJN_p
:= \overline{D_p(\mathbb{R}^n) \cap JN_p}^{JN_p},
\end{equation*}
where
\begin{equation*}
D_p(\mathbb{R}^n)
:= \{ f \in C^{\infty}(\mathbb{R}^n) : |\nabla f| \in L^p(\mathbb{R}^n) \}.
\end{equation*}
\end{definition}

\begin{definition}[$CJN_p$]
Let $1 < p < \infty$. Then the subspace $CJN_p$ of $JN_p$ is defined by setting
\begin{equation*}
CJN_p
:= \overline{C_c^{\infty}(\mathbb{R}^n) }^{JN_p},
\end{equation*}
where $C_c^{\infty}(\mathbb{R}^n)$ denotes the set of smooth functions with compact support in $\mathbb{R}^n$.
\end{definition}

As in the case of vanishing subspaces of $BMO$, there exist characterizations of $VJN_p$ and $CJN_p$ as $JN_p$ functions for which certain integrals vanish, see \cite[Theorems 3.2 and 4.3]{taoyangyuanvjnp}.
\begin{theorem}
Let $1 < p < \infty$. Then $f \in VJN_p$ if and only if $f \in JN_p$ and
\begin{equation*}
\lim_{a \to 0}
\sup_{\substack{Q_i \subset \mathbb{R}^n \\  l(Q_i) \leq a}}
\sum_{i=1}^{\infty} |Q_i| \left( \fint_{Q_i} | f - f_{Q_i} | \right)^p
= 0,
\end{equation*}
where the supremum is taken over all collections of pairwise disjoint cubes $(Q_i)_{i=1}^{\infty}$ in $\mathbb{R}^n$, such that the side length of each $Q_i$ is at most $a$.
\end{theorem}

\begin{theorem}
\label{cjnplause}
Let $1 < p < \infty$. Then $f \in CJN_p$ if and only if $f \in VJN_p$ and
\begin{equation*}
\lim_{a \to \infty}
\sup_{\substack{Q \subset \mathbb{R}^n \\  l(Q) \geq a}}
|Q|^{1/p} \fint_Q |f - f_Q|
= 0,
\end{equation*}
where the supremum is taken over all cubes $Q \subset \mathbb{R}^n$ such that the side length of $Q$ is at least $a$.
\end{theorem}

From the definitions we can see that $L^p / \mathbb{R} \subseteq CJN_p \subseteq VJN_p \subseteq JN_p$. It was shown in \cite{takala} that $L^p / \mathbb{R} \neq CJN_p$ and $VJN_p \neq JN_p$. However the question of whether $CJN_p$ and $VJN_p$ coincide remained open. 

\section{Equality of $VJN_p$ and $CJN_p$}

Let $1 < p < \infty$. In this section we prove the equality of $VJN_p$ and $CJN_p$ by showing that for any $JN_p$ function $f$, integrals of the type
\begin{equation*}
|Q|^{\frac{1}{p}} \fint_Q |f|
\end{equation*}
tend to zero both when $|Q| \to 0$ and when $|Q| \to \infty$.
This type of integral appears in the Morrey norm, see for example \cite{rafeirosamkosamko}. Compare it also to the weak $L^p$ norm (Definition \ref{weaklpdefinition}), where the supremum is taken over such integrals with the cube $Q$ replaced with an arbitrary measurable set.

%We call this a weak $L^p$ type integral, because the expression is similar to the weak $L^p$ norm, see Definition \ref{weaklpdefinition}.
%It is important to note that in Definition \ref{weaklpdefinition} the set $E$ is an arbitrary measurable set, but in this section $Q$ denotes a cube. 

The aforementioned results follow from Proposition \ref{paatulos}. 

\begin{proposition}
\label{paatulos}
Let $X \subseteq \mathbb{R}^n$ be either a bounded cube or the entire space $\mathbb{R}^n$. Let $Q \subset X$ be a cube such that $3 Q \subseteq X$. Let $1 < p < \infty$, $0 < A < \infty$ and $0 < \epsilon \leq \epsilon_0(n,p)$. Suppose that $f \in L_{loc}^1(X)$ is a nonnegative function such that
\begin{equation}
\label{firstinequality}
|Q|^{\frac{1}{p}} \fint_Q f
\geq A (1 - \epsilon)
\end{equation}
and for any cube $Q' \subset X$ with $l(Q') = \frac{2}{3} l(Q)$ or $l(Q') = \frac{4}{3} l(Q)$ we have
\begin{equation}
\label{secondinequality}
|Q'|^{\frac{1}{p}} \fint_{Q'} f
\leq A (1 + \epsilon).
\end{equation}
Then there exist two cubes $Q_1 \subset 3 Q$ and $Q_2 \subset 3Q$ such that for $i \in \{ 1 , 2\}$ we have
\begin{align}
\label{ekaehto}
&  l(Q_i) = \frac{2}{3} l(Q) \text{  or  } l(Q_i) = \frac{4}{3} l(Q),
\\
\label{tokaehto}
& \dist(Q_1 , Q_2) \geq \frac{1}{3} l(Q), \textrm{ and}
\\
& |Q_i|^{\frac{1}{p}} \fint_{Q_i} |f - f_{Q_i}|
\geq c \cdot A,
\label{kolmasehto}
\end{align}
where $c = c(n,p)$ is a positive constant.

\end{proposition}

To prove this proposition we first need to prove Lemmas \ref{ekalemma} and \ref{tokalemma}.

\begin{lemma}
\label{ekalemma}
Let $0 < \alpha < 1 < \beta$.
Let $Q' \subset Q \subset \tilde{Q} \subset \mathbb{R}^n$ be cubes such that $l(Q') = \alpha l(Q)$ and $l(\tilde{Q}) = \beta l(Q)$. Suppose that $f \in L^1(\tilde{Q})$ is a nonnegative function, $1 < p < \infty$, $0 < A < \infty$ and $0 < \epsilon \leq \epsilon_0(n , p , \alpha , \beta)$. Assume also that \eqref{firstinequality} holds for $Q$ and \eqref{secondinequality} holds for $Q'$ and $\tilde{Q}$.
Then we have
\begin{equation*}
|\tilde{Q} \setminus Q'|^{\frac{1}{p}} \fint_{\tilde{Q} \setminus Q'} \left| f - f_{\tilde{Q} \setminus Q'} \right|
\geq c_1 \cdot A,
\end{equation*}
where $c_1 = c_1(n , p , \alpha , \beta)$ is a positive constant.

\end{lemma}

\begin{proof}

From the assumptions of the lemma we get directly
\begin{align*}
\int_{\tilde{Q} \setminus Q'} \left| f - f_{\tilde{Q} \setminus Q'} \right|
%&= \int_{\tilde{Q} \setminus Q} \left| f - f_{\tilde{Q} \setminus Q'} \right| + \int_{Q \setminus Q'} \left| f - f_{\tilde{Q} \setminus Q'} \right|
%\\
&\geq \left| \int_{\tilde{Q} \setminus Q} f - \frac{|\tilde{Q} \setminus Q|}{|\tilde{Q} \setminus Q'|} \int_{\tilde{Q} \setminus Q'} f \right|
+ \left| \int_{Q \setminus Q'} f - \frac{|Q \setminus Q'|}{|\tilde{Q} \setminus Q'|} \int_{\tilde{Q} \setminus Q'} f \right|
\\
&= 2 \left| \int_Q f - \frac{1 - \alpha^n}{\beta^n - \alpha^n} \int_{\tilde{Q}} f - \frac{\beta^n - 1}{\beta^n - \alpha^n} \int_{Q'} f \right|
\\
\geq 2 \bigg( A (1 - \epsilon) |Q|^{1 - \frac{1}{p}} &- \frac{1 - \alpha^n}{\beta^n - \alpha^n} A (1 + \epsilon) |\tilde{Q}|^{1 - \frac{1}{p}} - \frac{\beta^n - 1}{\beta^n - \alpha^n} A (1 + \epsilon) |Q'|^{1 - \frac{1}{p}} \bigg)
\\
= 2 A |Q|^{1 - \frac{1}{p}} \bigg( 1 - \epsilon &- \frac{1 - \alpha^n}{\beta^n - \alpha^n} (1 + \epsilon) \left( \beta^n \right)^{1 - \frac{1}{p}} - \frac{\beta^n - 1}{\beta^n - \alpha^n} (1 + \epsilon) \left( \alpha^n \right)^{1 - \frac{1}{p}} \bigg).
\end{align*}
Thus we have
\begin{align*}
&| \tilde{Q} \setminus Q' |^{\frac{1}{p}} \fint_{\tilde{Q} \setminus Q'} \left|f - f_{\tilde{Q} \setminus Q'} \right|
%\\
%\geq\,& \left( \left( \beta^n - \alpha^n \right) |Q| \right)^{\frac{1}{p} - 1} 2 A |Q|^{1 - \frac{1}{p}} \cdot
%\bigg( 1 - \epsilon - %&
%\frac{1 - \alpha^n}{\beta^n - \alpha^n} (1 + \epsilon) \left( \beta^n \right)^{1 - \frac{1}{p}}
%- \frac{\beta^n - 1}{\beta^n - \alpha^n} (1 + \epsilon) \left( \alpha^n \right)^{1 - \frac{1}{p}} \bigg)
\\
&\quad\geq\, 2 \left( \beta^n - \alpha^n \right)^{\frac{1}{p} - 1} \bigg( 1 - \epsilon - \frac{1 - \alpha^n}{\beta^n - \alpha^n} (1 + \epsilon) \beta^{n - \frac{n}{p}}
- \frac{\beta^n - 1}{\beta^n - \alpha^n} (1 + \epsilon) \alpha^{n - \frac{n}{p}} \bigg) A
\\
&\quad =\, C(n , p , \alpha , \beta , \epsilon) A.
\end{align*}
Notice that
\begin{align*}
\lim_{\epsilon \to 0} C(n , p , \alpha , \beta , \epsilon)
&= C(n , p , \alpha , \beta , 0)
\\
&= 2 \left( \beta^n - \alpha^n \right)^{\frac{1}{p} - 1} \left( 1 - \frac{1 - \alpha^n}{\beta^n - \alpha^n} \beta^{n - \frac{n}{p}} - \frac{\beta^n - 1}{\beta^n - \alpha^n} \alpha^{n - \frac{n}{p}} \right).
\end{align*}
This is positive for every $n$, $p$, $\alpha $ and $\beta$. Indeed we have $C(n , p , \alpha , \beta , 0 ) = 2 (\beta^n - \alpha^n)^{\frac{1}{p} - 1} h(\frac{1}{p})$ with
\begin{equation*}
h(x) = 1 - \frac{1 - \alpha^n}{\beta^n - \alpha^n} \beta^{n - n x} - \frac{\beta^n - 1}{\beta^n - \alpha^n} \alpha^{n - n x}.
\end{equation*}
We notice that $h(0) = h(1) = 0$ and the second derivative of $h$ is strictly negative. Thus $h$ is concave and $h(x) > 0$ for every $0 < x < 1$.
In conclusion, if $\epsilon$ is small enough, we have
\begin{equation*}
C(n , p , \alpha , \beta , \epsilon)
\geq \frac{1}{2} C(n , p , \alpha , \beta , 0)
:= c_1 (n , p , \alpha , \beta)
> 0.
\end{equation*}
This completes the proof.
\end{proof}

For the reader's convenience, we start by giving a proof of Proposition~\ref{paatulos} in the special case $n=1$ as it is technically much simpler. The idea of the proof is the same also in the multidimensional case.

\begin{proof}[Proof of Proposition \ref{paatulos} in the case $n = 1$]

Let us assume that $Q = [a , a + L]$. Define $\tilde{Q} = \left[ a , a + \frac{4}{3} L \right]$ and $Q' = \left[ a , a + \frac{2}{3} L \right]$. We set $Q_1 := \tilde{Q} \setminus Q' = \left[ a + \frac{2}{3} L , a + \frac{4}{3} L \right]$ and we get from Lemma \ref{ekalemma} and the assumptions in Proposition \ref{paatulos} that
\begin{equation*}
|Q_1|^{\frac{1}{p}} \fint_{Q_1} |f - f_{Q_1}|
\geq c_1 A,
\end{equation*}
if $\epsilon$ is small enough. Here $c_1 = c_1(n,p,\alpha,\beta)$ with $n=1$, $\alpha = \frac{2}{3}$ and $\beta = \frac{4}{3}$.

On the other hand if we set $\tilde{Q} = \left[  a - \frac{1}{3} L , a + L \right] $ and $Q' = \big[ a + \frac{1}{3} L , a + L \big] $ and define $Q_2 := \tilde{Q} \setminus Q' = \left[ a - \frac{1}{3} L , a + \frac{1}{3} L \right]$, then we get from Lemma \ref{ekalemma} that
\begin{equation*}
|Q_2|^{\frac{1}{p}} \fint_{Q_2} |f - f_{Q_2}|
\geq c_1 A,
\end{equation*}
if $\epsilon$ is small enough.
Finally we notice that the distance between the cubes $Q_1$ and $Q_2$ is  $\frac{1}{3} L$. This completes the proof.

\end{proof}

The case $n \geq 2$ is more complicated as the set $\tilde{Q} \setminus Q'$ is usually not a cube. Before the actual proof, we fix some notation about directions and projections.

\begin{definition}\label{def:direction}
Let $\{v_1,\ldots,v_n\}$ denote the standard orthonormal basis for $\mathbb R^n$. Let $Q_1 \subset \mathbb{R}^n$ and $Q_2 \subset \mathbb{R}^n$ be cubes. The cubes can be presented as Cartesian products of intervals as
\begin{align*}
Q_1 &= I_1^1 \times I_2^1 \times ... \times I_n^1
\text{  and  }
\\
Q_2 &= I_1^2 \times I_2^2 \times ... \times I_n^2.
\end{align*}
We say that $P_k(Q_1) := I_k^1$ is the projection of cube $Q_1$ to the subspace spanned by the base vector $v_k$. Fix an index $k$ with $1\le k\le n$. If for every $x_k \in P_k(Q_1)$ and $y_k \in P_k(Q_2)$ we have $x_k \leq y_k$, then we say that $Q_2$ is located in direction $\uparrow^k$ from $Q_1$ and $Q_1$ is located in direction $\downarrow^k$ from $Q_2$.
\end{definition}

\begin{lemma}
\label{tokalemma}
Let $Q \subset \mathbb{R}^n$ be a cube. Let $Q' \subset Q \subset \tilde{Q}$ be cubes with $l(Q') = \frac{2}{3} l(Q)$ and $l(\tilde{Q}) = \frac{4}{3} l(Q)$ such that all the cubes share a corner.
%In other words if
%\begin{equation*}
%Q = \prod_{i=1}^n \left[ a_i , a_i + L \right] ,
%\end{equation*}
%\begin{equation*}
%Q' = \prod_{i=1}^n \left[ b_i , b_i + \frac{2}{3} L \right]
%\end{equation*}
%and
%\begin{equation*}
%\tilde{Q} = \prod_{i=1}^n \left[ c_i , c_i + \frac{4}{3} L \right] ,
%\end{equation*}
%then for each $i$ we have either $a_i = b_i = c_i$ or $a_i + L = b_i + \frac{2}{3} L = c_i + \frac{4}{3} L$.
By symmetry we may assume that $Q = \left[ 0 , L \right]^n$, $Q' = \left[ 0 , \frac{2}{3} L \right]^n$ and $\tilde{Q} = \left[ 0 , \frac{4}{3} L \right]^n$.
Let $f \in L^1 \left( [0 , 2L]^n \right)$ be a nonnegative function, $1 < p < \infty$, $0 < A < \infty$ and $0 < \epsilon \leq \epsilon_0$ from Lemma \ref{ekalemma} with $\alpha = \frac{2}{3}$ and $\beta = \frac{4}{3}$.
Suppose also that \eqref{firstinequality} holds for $Q$ and \eqref{secondinequality} holds for $Q'$ and $\tilde{Q}$.
Then there exists a cube $\bar{Q} \subset [0, 2L]^n \setminus Q'$ such that either
\begin{itemize}
\item[a)]
$l(\bar{Q}) = \frac{2}{3} L$ and $\bar{Q} \subset \tilde{Q}$
\\
or
\item[b)]
$l(\bar{Q}) = \frac{4}{3} L$ and $P_k(\bar{Q}) = P_k(\tilde{Q}) = [0 , \frac{4}{3} L]$ for every $1 \leq k \leq n$ except one
\end{itemize}
and in addition
\begin{equation*}
|\bar{Q}|^{\frac{1}{p}} \fint_{\bar{Q}} |f - f_{\bar{Q}}|
\geq c_2 \cdot A,
\end{equation*}
where $c_2 = c_2(n,p)$ is a positive constant.
\end{lemma}

\begin{proof}

Let us divide the cube $\tilde{Q}$ dyadically into $2^n$ subcubes. Then one of them is $Q'$. Let us define $Q_0 = \left[ \frac{2}{3} L , \frac{4}{3} L \right]^n$ and let us denote the rest of the subcubes by $(Q_i)_{i=1}^{2^n - 2}$. Notice that $Q'$ does not have an index unlike all the other dyadic subcubes.
For any $1 \leq j \leq n$ we define
\begin{equation*}
Q_j' := I_1 \times I_2 \times ... \times I_n,
\end{equation*}
where
\[
I_k=
\begin{cases}
[\frac{2}{3} L , 2L],& k=j,\\
[0 , \frac{4}{3} L], & k\neq j.\\
\end{cases}
\]
See Figure \ref{tokanlemmankuva} to see how these cubes are located with respect to each other.
It is simple to check that then 
\[
Q_0 \subset Q_j' \subset [0 , 2L]^n \setminus Q' \quad\textrm{ and }\quad l(Q_j') = 2 l(Q_0)
\]
for every $j$. Also for every $Q_i$ with $1\le i\le 2^n-2$, there exists at least one cube $Q_j'$ such that $Q_i \subset Q_j'$.
For every $Q_i$ let us denote by $Q_{j_i}'$ one of these cubes $Q_j'$. 
%There might be multiple such cubes, but it is enough that we can find one.

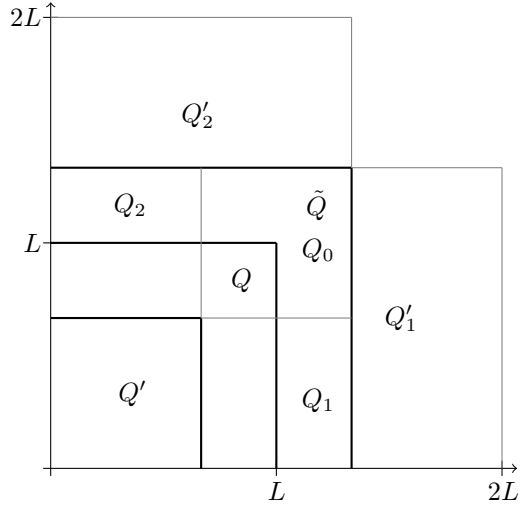
\begin{figure}[ht]
\centering
\begin{tikzpicture}

\draw[black,->] (-0.1,0) -- (6.2,0);
\draw[black,->] (0,-0.1) -- (0,6.2) ;

\draw (-0.1,3)--(0.1,3);
\draw (0,3) node[left] {$L$};
\draw (-0.1,6)--(0.1,6);
\draw (0,6) node[left] {$2L$};
\draw (3,-0.1)--(3,0.1);
\draw (3.25,-0.3) node[left] {$L$};
\draw (6,-0.1)--(6,0.1);
\draw (6.35,-0.3) node[left] {$2L$};

% Q
\draw[black, thick] (0,3)--(3,3);
\draw[black, thick] (3,0)--(3,3);
\draw (2.8,2.5) node[left] {$Q$};

% Q'
\draw[black, thick] (0,2)--(2,2);
\draw[black, thick] (2,0)--(2,2);
\draw (1.4,1) node[left] {$Q'$};

% \tilde{Q}
\draw[black, thick] (0,4)--(4,4);
\draw[black, thick] (4,0)--(4,4);
\draw (3.8,3.5) node[left] {$\tilde{Q}$};

% Q_0, Q_1, Q_2
\draw[gray] (2,2)--(2,4);
\draw[gray] (2,2)--(4,2);
\draw (3.9,2.9) node[left] {$Q_0$};
\draw (3.9,0.9) node[left] {$Q_1$};
\draw (1.4,3.5) node[left] {$Q_2$};

% Q_1', Q_2'
\draw[gray] (4,4)--(4,6);
\draw[gray] (4,4)--(6,4);
\draw[gray] (0,6)--(4,6);
\draw[gray] (6,0)--(6,4);
\draw (5,2) node[left] {$Q_1'$};
\draw (2.3,4.7) node[left] {$Q_2'$};

\end{tikzpicture}

\caption{The cubes $Q$, $Q'$, $\tilde{Q}$, $(Q_i)_{i=0}^{2^n-2}$ and $(Q_j')_{j=1}^n$, when $n = 2$. We have $\tilde{Q} = Q' \cup Q_0 \cup Q_1 \cup Q_2$, $Q_0 \cup Q_1 \subset Q_1'$ and $Q_0 \cup Q_2 \subset Q_2'$.}

\label{tokanlemmankuva}
\end{figure}

Let us prove that for at least one of the cubes $Q_i$ or $Q_j'$ we have
\begin{equation}
\label{vastaoletus}
|\bar{Q}|^{\frac{1}{p}} \fint_{\bar{Q}} |f - f_{\bar{Q}}|
\geq \left( 2^n - 1 \right)^{-\frac{1}{p}} \left( 1 + 2^{1 + n - \frac{n}{p}} \left( \frac{2^n - 2}{2^n - 1} \right)^2 \right)^{-1} c_1 A,
\end{equation}
where $c_1 = c_1(n , p , \alpha , \beta)$ is the constant from Lemma \ref{ekalemma} with $\alpha = \frac{2}{3}$ and $\beta = \frac{4}{3}$.
We prove this by contradiction. Assume that \eqref{vastaoletus} does not hold for any  $Q_i$ or $Q_j'$. We get from Lemma \ref{ekalemma} that
\begin{align}
\nonumber
c_1 A
&\leq |\tilde{Q} \setminus Q'|^{\frac{1}{p}} \fint_{\tilde{Q} \setminus Q'} \left| f - f_{\tilde{Q} \setminus Q'} \right|
\\
&= \left( (2^n - 1) |Q'| \right)^{\frac{1}{p} - 1} \cdot \sum_{i=0}^{2^n - 2} \int_{Q_i} \left| f - \frac{1}{(2^n - 1)|Q'|} \sum_{k=0}^{2^n-2} \int_{Q_k} f \right|.
\label{sumtobeestimated}
\end{align}
We continue estimating one of the integrals in the sum above
\begin{align*}
&\int_{Q_i} \Bigg| f - \frac{1}{(2^n - 1)|Q'|} \sum_{k=0}^{2^n-2} \int_{Q_k} f \Bigg|
\\
= &\int_{Q_i} \Bigg| f - f_{Q_i} + \frac{1}{(2^n - 1)|Q'|} \sum_{k=0}^{2^n - 2} \left( \int_{Q_i} f - \int_{Q_k} f \right) \Bigg|
\\
\leq &\int_{Q_i} |f - f_{Q_i}| + \frac{1}{2^n - 1} \sum_{k=0}^{2^n - 2} \left| \int_{Q_i} f -  \int_{Q_k} f \right|.
\end{align*}
Assume that $k \geq 1$. Then because $Q_k \cup Q_0 \subset Q_{j_k}'$ and $Q_k \cap Q_0 = \varnothing$, we have
\begin{equation*}
\left| \int_{Q_k} f -  \int_{Q_0} f \right|
%&= \left| \int_{Q_k} f - |Q_k| f_{Q_{j_k}'} + |Q_0| f_{Q_{j_k}'} - \int_{Q_0} f \right|
\leq \left| \int_{Q_k} \left( f - f_{Q_{j_k}'} \right) \right| + \left| \int_{Q_0} \left( f_{Q_{j_k}'} - f \right) \right|
\leq \int_{Q_{j_k}'} \left| f - f_{Q_{j_k}'} \right|.
\end{equation*}
Thus if $i = 0$, we get
\begin{equation*}
\sum_{k=0}^{2^n - 2} \left| \int_{Q_i} f -  \int_{Q_k} f \right|
\leq \sum_{k=1}^{2^n - 2} \int_{Q_{j_k}'} \left| f - f_{Q_{j_k}'} \right|.
\end{equation*}
On the other hand if $i \geq 1$, then
\begin{align*}
\sum_{\substack{k=0 \\ k \neq i }}^{2^n - 2} \left| \int_{Q_i} f -  \int_{Q_k} f \right|
&\leq (2^n - 2) \left| \int_{Q_i} f - \int_{Q_0} f \right| + \sum_{\substack{k=0 \\ k \neq i}}^{2^n - 2} \left| \int_{Q_0} f - \int_{Q_k} f \right|
\\
&\leq (2^n - 2) \int_{Q_{j_i}'} \left| f - f_{Q_{j_i}'} \right| + \sum_{\substack{k=1 \\ k \neq i}}^{2^n - 2} \int_{Q_{j_k}'} \left| f - f_{Q_{j_k}'} \right|.
\end{align*}
We continue by estimating the sum in \eqref{sumtobeestimated} and we get
\begin{align*}
&\sum_{i=0}^{2^n - 2} \int_{Q_i} \left| f - \frac{1}{(2^n - 1)|Q'|} \sum_{k=0}^{2^n-2} \int_{Q_k} f \right|
\\
\leq &\sum_{i=0}^{2^n - 2} \left( \int_{Q_i} |f - f_{Q_i}| + \frac{1}{2^n - 1} \sum_{k=0}^{2^n - 2} \left| \int_{Q_i} f -  \int_{Q_k} f \right| \right)
\\
\leq &\sum_{i=0}^{2^n - 2} \int_{Q_i} |f - f_{Q_i}| + \frac{1}{2^n - 1} \sum_{k=1}^{2^n - 2} \int_{Q_{j_k}'} \left| f - f_{Q_{j_k}'} \right|
\\
&\quad\quad+ \sum_{i=1}^{2^n - 2} \left( \frac{1}{2^n - 1} \left( (2^n - 2) \int_{Q_{j_i}'} \left| f - f_{Q_{j_i}'} \right| + \sum_{\substack{k=1 \\ k \neq i}}^{2^n - 2} \int_{Q_{j_k}'} \left| f - f_{Q_{j_k}'} \right| \right) \right)
\\
= &\sum_{i=0}^{2^n - 2} \int_{Q_i} |f - f_{Q_i}| + \sum_{k=1}^{2^n - 2} \int_{Q_{j_k}'} \left| f - f_{Q_{j_k}'} \right| + \frac{1}{2^n - 1} \sum_{i=1}^{2^n - 2} \sum_{\substack{k=1 \\ k \neq i}}^{2^n - 2} \int_{Q_{j_k}'} \left| f - f_{Q_{j_k}'} \right|
\\
= &\sum_{i=0}^{2^n - 2} \int_{Q_i} |f - f_{Q_i}| + \left( 1 + \frac{2^n - 3}{2^n - 1} \right) \sum_{k=1}^{2^n - 2} \int_{Q_{j_k}'} \left| f - f_{Q_{j_k}'} \right| .
\end{align*}
Then finally from our assumption, that \eqref{vastaoletus} does not hold, we get
\begin{align*}
&\sum_{i=0}^{2^n - 2} \int_{Q_i} |f - f_{Q_i}| + \left( 1 + \frac{2^n - 3}{2^n - 1} \right) \sum_{k=1}^{2^n - 2} \int_{Q_{j_k}'} \left| f - f_{Q_{j_k}'} \right|
\\
<& \sum_{i=0}^{2^n - 2} \left( 2^n - 1 \right)^{-\frac{1}{p}} \left( 1 + 2^{1 + n - \frac{n}{p}} \left( \frac{2^n - 2}{2^n - 1} \right)^2 \right)^{-1} c_1 A |Q_i|^{1 - \frac{1}{p}}
\\
&\quad+ 2 \cdot \frac{2^n - 2}{2^n - 1} \sum_{k=1}^{2^n - 2} \left( 2^n - 1 \right)^{-\frac{1}{p}} \left( 1 + 2^{1 + n - \frac{n}{p}} \left( \frac{2^n - 2}{2^n - 1} \right)^2 \right)^{-1} c_1 A |Q_{j_k}'|^{1 - \frac{1}{p}}
\\
=& \left( 2^n - 1 \right)^{1-\frac{1}{p}} \left( 1 + 2^{1 + n - \frac{n}{p}} \bigl( \tfrac{2^n - 2}{2^n - 1} \bigr)^2 \right)^{-1} c_1 A \biggl( |Q'|^{1 - \frac{1}{p}}
+  \tfrac{2(2^n - 2)^2}{(2^n - 1)^2} \left( 2^n |Q'| \right)^{1 - \frac{1}{p}} \biggr)
\\
=& (2^n - 1)^{1 - \frac{1}{p}} c_1 A |Q'|^{1 - \frac{1}{p}}.
\end{align*}
In conclusion we have
\begin{align*}
c_1 A
&\leq |\tilde{Q} \setminus Q'|^{\frac{1}{p}} \fint_{\tilde{Q} \setminus Q'} \left| f - f_{\tilde{Q} \setminus Q'} \right|
\\
&< \left( (2^n - 1) |Q'| \right)^{\frac{1}{p} - 1} \cdot (2^n - 1)^{1 - \frac{1}{p}} c_1 A |Q'|^{1 - \frac{1}{p}}
= c_1 A,
\end{align*}
which is a contradiction. Hence there is at least one cube $\bar{Q} \subset [0 , 2L]^n \setminus Q'$ that satisfies the conditions of the lemma and
\begin{equation*}
|\bar{Q}|^{\frac{1}{p}} \fint_{\bar{Q}} |f - f_{\bar{Q}}|
\geq \left( 2^n - 1 \right)^{-\frac{1}{p}} \left( 1 + 2^{1 + n - \frac{n}{p}} \left( \frac{2^n - 2}{2^n - 1} \right)^2 \right)^{-1} c_1 A
= c_2 A,
\end{equation*}
where $c_2 = c_2(n,p)$ is a positive constant. This completes the proof.
\end{proof}

Now we are ready to prove Proposition \ref{paatulos}.

\begin{proof}[Proof of Proposition \ref{paatulos} in the case $n \geq 2$]

Let $0 < \epsilon \leq \epsilon_0$ from Lemma \ref{ekalemma} with $\alpha = \frac{2}{3}$ and $\beta = \frac{4}{3}$. We apply Lemma \ref{tokalemma} $2^n$ times to each corner of $Q$ and obtain $2^n$ sets of cubes $Q'$, $\tilde{Q}$ and $\bar{Q}$. Here every $\bar{Q}$ satisfies \eqref{ekaehto} and \eqref{kolmasehto}.
No matter which corner the cubes $Q'$, $Q$ and $\tilde{Q}$ share, we always have $\frac{1}{3} Q \subset Q'$. Since $\bar{Q}$ and $Q'$ are disjoint, we get that $\bar{Q}$ and $\frac{1}{3} Q$ are also disjoint and thus each $\bar{Q}$ is located in at least one direction from $\frac{1}{3} Q$ in the sense of Definition \ref{def:direction}.

Because $Q'$ is in the corner of $\tilde{Q}$, there is one direction for each $k \in \{1 , 2 , ... , n\}$ such that $\bar{Q}$ cannot be located in that direction from $\frac{1}{3} Q$. For example if $Q = [0,L]^n$, $Q' = [0, \frac{2}{3} L]^n$ and $\tilde{Q} = [0, \frac{4}{3} L]^n$, then the possible directions where $\bar{Q}$ may be located in from $\frac{1}{3} Q$ are $\uparrow^1$, $\uparrow^2$, ... and $\uparrow^n$. The cube $\bar{Q}$ cannot be located in any of the directions $\downarrow^1$, $\downarrow^2$, ... and $\downarrow^n$ from $\frac{1}{3} Q$. Thus for each cube $\bar{Q}$ there are $n$ possible directions and for any two cubes $\bar{Q}$ the sets of possible directions do not coincide.

If one cube $\bar{Q}$ is located in direction $\uparrow^k$ from $\frac{1}{3} Q$ and another is located in direction $\downarrow^k$, then the distance between those two cubes is at least $\frac{1}{3} l(Q)$ -- thus the proposition is true. Therefore let us assume by contradiction that no two cubes $\bar{Q}$ are located in opposite directions.

Let $S$ be the set of all directions in which all the cubes $\bar{Q}$ are located from $\frac{1}{3} Q$. Clearly then we have $|S| \leq n$, because by our assumption there is at most one direction in $S$ for each $k \in \{1 , 2 , ... , n\}$. However there is always at least one cube $\bar{Q}$ for which the possible directions are all exactly opposite to the directions in $S$. If for example the directions in $S$ are $\uparrow^1$, $\uparrow^2$, ... and $\uparrow^m$, then there is no direction in $S$ for the cube $\bar{Q}$ for which the possible directions are only $\downarrow^1$, $\downarrow^2$, ... and $\downarrow^n$. This contradicts with the assumption that the directions of all $\bar{Q}$ are represented in $S$. Thus we conclude that there must be two cubes $\bar{Q}$ in opposite directions. This completes the proof.
\end{proof}

Now we can show that the Morrey type integral \eqref{eqn:morrey} vanishes as the measure of the cube tends to infinity.

\begin{theorem}
\label{vanishingbigcubes}
Let $1 < p < \infty$ and suppose that $f \in JN_p(\mathbb{R}^n)$. Then there is a constant $b$ such that
\begin{equation}
\label{vanishingbigcubesequation}
\lim_{a \to \infty} \sup_{\substack{Q \subset \mathbb{R}^n \\ l(Q) \geq a }} |Q|^{1/p} \fint_Q |f-b| = 0,
\end{equation}
where the supremum is taken over all cubes $Q \subset \mathbb{R}^n$ such that the side length of $Q$ is at least $a$.
\end{theorem}

\begin{proof}

Assume first that $f \in JN_p(\mathbb{R}^n) \cap L^{p,\infty}(\mathbb{R}^n)$ and $f$ is nonnegative.
Let
\begin{equation*}
\lim_{a \to \infty} \sup_{\substack{Q \subset \mathbb{R}^n \\ l(Q) \geq a }} |Q|^{1/p} \fint_Q f
= A,
\end{equation*}
where $A \geq 0$. The limit exists as this sequence is decreasing, the elements are finite -- this follows from Remark \ref{weaklpremark} because $f \in L^{p,\infty}(\mathbb{R}^n)$ -- and the sequence is bounded from below by 0.

Let us assume that $A > 0$.
Let $0 < \epsilon \leq \epsilon_0(n , p)$ as in Proposition \ref{paatulos}. Then there exists a number $N < \infty$ such that \eqref{secondinequality} holds for every cube $Q$ where $l(Q) \geq N$. Also for any number $M < \infty$ there exists a cube $Q$ such that $l(Q) \geq M$ and \eqref{firstinequality} holds.
This means that we can find a sequence of cubes $(Q_i)_{i=1}^{\infty}$ such that $l(Q_1) \geq \frac{3}{2} N$, $l(Q_{i+1}) > l(Q_i)$, $\lim_{i \to \infty} l(Q_i) = \infty$ and
\begin{equation*}
|Q_i|^{1/p} \fint_{Q_i} f
\geq A (1 - \epsilon)
\end{equation*}
for every $i \in \mathbb{N}$.

Let $Q_i$ be one of these cubes. According to Proposition \ref{paatulos} there exist two cubes $Q_{i,1} \subset 3 Q_i$ and $Q_{i,2} \subset 3 Q_i$ that satisfy \eqref{ekaehto}, \eqref{tokaehto} and \eqref{kolmasehto}.
The cubes $(Q_{i,1})_{i=1}^{\infty}$ may not be pairwise disjoint. However for every cube $Q_i$ we have two cubes to choose from.

Let us construct a new sequence of cubes $(Q_{{i_j}}')_{j=1}^{\infty}$ iteratively.
We start with $Q_{i_1} := Q_1$ and choose $Q_{i_1}' := Q_{1,1}$.
Let $Q_{i_2}$, $i_2 > 1$, be the smallest cube in the sequence $(Q_i)_{i=1}^{\infty}$ such that $\frac{1}{3} l(Q_{i_2}) \geq l(Q_{i_1}')$.
Then at least one of the cubes $Q_{i_2,1}$ and $Q_{i_2,2}$ is pairwise disjoint with $Q_{i_1}'$.
Let's say that $Q_{i_2,1}$ is the disjoint one and set $Q_{i_2}' := Q_{i_2,1}$.

Let us denote by $Q$ the smallest cube such that $Q_{i_1}' \cup Q_{i_2}' \subset Q$. Let $Q_{i_3}$, $i_3 > i_2$, be the smallest cube in the sequence $(Q_i)_{i=1}^{\infty}$ such that $\frac{1}{3} l(Q_{i_3}) \geq l(Q)$. Then at least one of the cubes $Q_{i_3 , 1}$ and $Q_{i_3 , 2}$ is pairwise disjoint with both $Q_{i_1}'$ and $Q_{i_2}'$.

By repeating this process and taking a subsequence of $(Q_i)_{i=1}^{\infty}$, if necessary, we get infinitely many pairwise disjoint cubes $(Q_{i_j}')_{j=1}^{\infty}$. Then we get
\begin{equation*}
\sum_{j=1}^{\infty} |Q_{i_j}'| \left( \fint_{Q_{i_j}'} \left| f - f_{Q_{i_j}'} \right| \right)^p
\geq \sum_{j=1}^{\infty} c(n,p) A^p
= \infty.
\end{equation*}
This contradicts with the assumption that $f \in JN_p(\mathbb{R}^n)$. Thus we conclude that $A = 0$.

Now assume only that $f \in JN_p(\mathbb{R}^n)$. From \cite[Theorem 5.2]{takala} we get that there is a constant $b$ such that $f - b \in L^{p,\infty}(\mathbb{R}^n)$. Obviously $f - b \in JN_p(\mathbb{R}^n)$. Then from \eqref{itseisarvoepayhtalo} we get that $|f-b| \in JN_p(\mathbb{R}^n) \cap L^{p,\infty}(\mathbb{R}^n)$.
Finally by the same reasoning as before, we conclude that
\begin{equation*}
\lim_{a \to \infty} \sup_{\substack{Q \subset \mathbb{R}^n \\ l(Q) \geq a }} |Q|^{1/p} \fint_Q |f-b|
= 0.
\end{equation*}
This completes the proof.
\end{proof}

\begin{remark}
\label{firstremark}
Naturally Theorem \ref{vanishingbigcubes} implies that the result holds also for $L^p$ functions with $p>1$, but this can also be seen easily by considering the Hardy-Littlewood maximal function.
Indeed assume by contradiction that
\begin{equation*}
\lim_{a \to \infty} \sup_{\substack{Q \subset \mathbb{R}^n \\ l(Q) \geq a}} |Q|^{\frac{1}{p}} \fint_Q |f| = A > 0.
\end{equation*}
Then for any number $a < \infty$ there exists a cube $Q$ such that $l(Q) \geq a$ and
\begin{equation*}
|Q|^{\frac{1}{p}} \fint_Q |f| \geq \frac{A}{2}.
\end{equation*}
Thus for every $x \in Q$ we have $M f(x) \geq \frac{A}{2} |Q|^{- \frac{1}{p}}$, where $Mf$ is the non-centered Hardy-Littlewood maximal function. Thus $Mf(x) \notin L^p$ and consequently $f \notin L^p$.

%Because $a$ can be arbitrarily large, it follows that $\| Mf \|_{L^p(\mathbb{R}^n)} = \infty$, which contradicts with the fact that $f \in L^p(\mathbb{R}^n)$. 
%Thus we conclude that
%\begin{equation*}
%\lim_{a \to \infty} \sup_{\substack{Q %\subset \mathbb{R}^n \\ l(Q) \geq a}} %|Q|^{\frac{1}{p}} \fint_Q |f| = 0.
%\end{equation*}
\end{remark}

Clearly the theorem does not hold for weak $L^p$ functions. For example let $n = 1$ and $f(x) = x^{-1/p}$, when $x > 0$. Then we have $f \in L^{p,\infty}(\mathbb{R})$ but \eqref{vanishingbigcubesequation} does not hold for any $b$.
The same function shows us that Theorem \ref{vanishingsmallcubes} does not hold for weak $L^p$ functions.

As a corollary from Theorem \ref{vanishingbigcubes} we get that $VJN_p$ and $CJN_p$ coincide. This answers a question that was posed in \cite{takala} and it answers \cite[Question 5.6]{taoyangyuanvjnp} and \cite[Question 17]{taoyangyuansurvey}.

\begin{corollary}
\label{cjnpvjnptulos}
Let $1 < p< \infty$. Then $CJN_p(\mathbb{R}^n) = VJN_p(\mathbb{R}^n)$.
\end{corollary}
\begin{proof}
It is clear that $CJN_p \subseteq VJN_p$.
Let $f \in VJN_p \subset JN_p$. Then we get from Theorem \ref{vanishingbigcubes} that there is a constant $b$ such that
\begin{align*}
\lim_{a \to \infty}
\sup_{\substack{Q \subset \mathbb{R}^n \\  l(Q) \geq a}}
|Q|^{1/p} \fint_Q |f - f_Q|
&= \lim_{a \to \infty}
\sup_{\substack{Q \subset \mathbb{R}^n \\  l(Q) \geq a}}
|Q|^{1/p} \fint_Q |f - b - (f - b)_Q|
\\
&\leq \lim_{a \to \infty}
\sup_{\substack{Q \subset \mathbb{R}^n \\  l(Q) \geq a}}
|Q|^{1/p} 2 \fint_Q |f - b|
= 0.
\end{align*}
Then by Theorem \ref{cjnplause} we get that $f \in CJN_p$. This completes the proof.
\end{proof}

From the following result we can infer that the additional condition in \cite[Lemma 5.8]{takala} is not necessary, answering the question posed in \cite{takala}.

\begin{theorem}
\label{vanishingsmallcubes}
Let $X \subseteq \mathbb{R}^n$ be either a bounded cube or the entire space $\mathbb{R}^n$ and $1 < p < \infty$. Suppose that $f \in JN_p(X)$. Then
\begin{equation}
\label{vanishingsmallcubesequation}
\lim_{a \to 0} \sup_{\substack{ Q \subseteq X \\ l(Q) \leq a }} |Q|^{1/p} \fint_Q |f| = 0,
\end{equation}
where the supremum is taken over all cubes $Q \subseteq X$, such that the side length of $Q$ is at most $a$.
\end{theorem}

\begin{remark}
\label{secondremark}
The theorem implies that $L^p$ functions also satisfy \eqref{vanishingsmallcubesequation}. However in that case the result is well known as it follows simply from H\"older's inequality and the dominated convergence theorem. Notice that for $L^p$ functions \eqref{vanishingsmallcubesequation} holds also when $p = 1$, whereas in Remark \ref{firstremark} it is necessary that $p > 1$.
%Assume that $f \in L^p(X) / \mathbb{R}$. Then there is a constant $b$ such that $f - b \in L^p(X)$. Clearly for any $m \in \mathbb{N}$ there is a cube $Q_m \subset X$ such that $l(Q_m) \leq 2^{-m}$ and
%\begin{equation*}
%|Q_m|^{1/p} \fint_{Q_m} |f - b|
%\geq \frac{1}{2} \sup_{\substack{ Q \subseteq X \\ l(Q) \leq 2^{-m} }} |Q|^{1/p} \fint_Q |f - b|.
%\end{equation*}
%Notice also that the sequence of functions $(f-b)_m := (f-b) \cdot \chi_{Q_m}$ converges to 0 pointwise almost everywhere, when $m \to \infty$.
%Thus we get from H\"older's inequality and the dominated convergence theorem that
%\begin{align*}
%\left( \lim_{a \to 0} \sup_{\substack{ Q \subseteq X \\ l(Q) \leq a }} |Q|^{1/p} \fint_Q |f| \right)^p
%&= \lim_{m \to \infty} \left( \sup_{\substack{ Q \subseteq X \\ l(Q) \leq 2^{-m} }} |Q|^{1/p} \fint_Q |f| \right)^p
%\\
%\leq \lim_{m \to \infty} &\left( \sup_{\substack{ Q \subseteq X \\ l(Q) \leq 2^{-m} }} |Q|^{1/p} \fint_Q |f - b| + \sup_{\substack{ Q \subseteq X \\ l(Q) \leq 2^{-m} }} |Q|^{1/p} \fint_Q |b| \right)^p
%\\
%&\leq \lim_{m \to \infty} \left( 2 |Q_m|^{1/p} \fint_{Q_m} |f - b| + 2^{- \frac{mn}{p}} |b| \right)^p
%\\
%&= \lim_{m \to \infty} \left( 2 |Q_m|^{1/p} \fint_{Q_m} |f - b| \right)^p
%\\
%&\leq \lim_{m \to \infty} 2^p \int_{Q_m} |f - b|^p
%= 2^p \int_X \lim_{m \to \infty} |(f - b)_m|^p
%\\
%&= 0.
%\end{align*}

\end{remark}

\begin{remark}
In particular this result implies that for any $1 < p < \infty$ and bounded cube $X \subset\mathbb{R}^n$, the John-Nirenberg space $JN_p(X)$ is a subspace of the vanishing Morrey space $VL^{1,n-\frac{n}{p}}(X)$ as defined in \cite{rafeirosamkosamko}.
\end{remark}

\begin{proof}[Proof of Theorem \ref{vanishingsmallcubes}]
Otherwise the proof is similar to the proof of Theorem~\ref{vanishingbigcubes}, but we have to take into account the possibility that we might have $3 Q \not\subseteq X$, even though $Q \subseteq X$.

Let $f \in JN_p(X)$. Let us assume first that $f$ is nonnegative and $f \in L^{p,\infty}(X)$. Let
\begin{equation}
\label{unvanishingsmallcubes}
\lim_{a \to 0} \sup_{\substack{ Q \subseteq X \\ l(Q) \leq a }} |Q|^{1/p} \fint_Q f = A.
\end{equation}
The limit exists as the sequence is decreasing and bounded from below by 0.

Assume that $A > 0$.
Let $0 < \epsilon \leq \epsilon_0(n,p)$ as in Proposition \ref{paatulos}. Then there is a number $\delta > 0$ such that \eqref{secondinequality} holds for any cube $Q$ where $l(Q) \leq \delta$. Also because of \eqref{unvanishingsmallcubes} we know that for every $a > 0$ there exists a cube $Q \subseteq X$ such that $l(Q) \leq a$ and \eqref{firstinequality} holds for $Q$.
Therefore we can find a sequence of cubes $(Q_i)_{i=1}^{\infty}$ such that $l(Q_1) \leq \frac{3}{4} \delta$, $l(Q_{i+1}) < l(Q_i)$, $\lim_{i \to \infty} l(Q_i) = 0$ and
\begin{equation*}
|Q_i|^{1/p} \fint_{Q_i} f \geq A (1 - \epsilon).
\end{equation*}
for every $i \geq 1$.

{\bf Case 1:} There are infinitely many cubes $Q_i$ in the sequence such that $3 Q_i \subseteq X$.

By taking a subsequence we may assume that $3 Q_i \subseteq X$ for every $i$.
Let $Q_i$ be one of these cubes. Then according to Proposition \ref{paatulos} there exist two cubes $Q_{i,1} \subset 3 Q_i$ and $Q_{i,2} \subset 3 Q_i$ that satisfy \eqref{ekaehto}, \eqref{tokaehto} and \eqref{kolmasehto}.
The cubes $(Q_{i,1})_{i=1}^{\infty}$ may not be pairwise disjoint. However for every cube $Q_i$ we have two cubes to choose from.

Let us construct a new sequence of cubes $(Q_{i_j}')_{j=1}^{\infty}$ iteratively. We start with $Q_{i_1} := Q_1$. Then there exist infinitely many cubes $Q_i$ in the sequence such that $Q_{1,1} \cap 3 Q_i = \varnothing$ for every $i$ or $Q_{1,2} \cap 3 Q_i = \varnothing$ for every $i$. This is true because $\dist(Q_{1,1} , Q_{1,2}) \geq l(Q_1) / 3$ and $\lim_{i \to \infty} l(Q_i) = 0$.
Without loss of generality we may assume that $Q_{1,1} \cap 3 Q_i = \varnothing$ for infinitely many $i > 1$. We set $Q_{i_1}' := Q_{1,1}$. Let $Q_{i_2}$, $i_2 > 1$, be the first cube in the sequence $(Q_i)_{i=1}^{\infty}$ such that $Q_{i_1}' \cap 3 Q_{i_2} = \varnothing$. Then the cubes $Q_{i_1}'$, $Q_{i_2 , 1}$ and $Q_{i_2 , 2}$ are all pairwise disjoint. Thus we can continue by choosing one of the cubes $Q_{i_2 , 1}$ and $Q_{i_2 , 2}$ such that there are still infinitely many cubes $3 Q_i$, $i > i_2$, that are pairwise disjoint with both the chosen cube and $Q_{i_1}'$.

By repeating this process and taking a subsequence of $(Q_i)_{i=1}^{\infty}$, if necessary, we get infinitely many pairwise disjoint cubes $(Q_{i_j}')_{j=1}^{\infty}$ in $X$ that all satisfy \eqref{kolmasehto}.

{\bf Case 2:} There are only finitely many cubes $Q_i$ in the sequence such that $3 Q_i \subseteq X$.
This can only happen if $X$ is a bounded cube.
In this case we shall also construct a new sequence of pairwise disjoint cubes $(Q_{i_j}')_{j=1}^{\infty}$ iteratively.
By taking a subsequence of $(Q_i)_{i=1}^{\infty}$ we may assume that $l(Q_1) \leq \frac{1}{4} l(X)$ and $l(Q_{i+1}) \leq \frac{2}{9} l(Q_i)$. Assume that $3 Q_i \not\subset X$.
Then we know that for any $k \in \{1 , 2 , ... , n\}$ the projection $P_k(3 Q_i)$ can only intersect one of the endpoints of $P_k(X)$.
Let $m$ be the number of base vectors $v_k$ such that $P_k(3 Q_i) \subset P_k(X)$.
Then $0 \leq m < n$ and there exist $2^m$ cubes $\bar{Q}_i \subset X \cap 3 Q_i \setminus \frac{1}{3} Q_i$ as they are defined in Lemma \ref{tokalemma}.

Every such cube $\bar{Q}_i$ is located in some direction from the cube $\frac{1}{3} Q_i$, in a similar way as in the proof of Proposition \ref{paatulos}. Without loss of generality we may assume that the possible directions for these cubes are $\uparrow^1$, $\downarrow^1$, $\uparrow^2$, $\downarrow^2$, ..., $\uparrow^m$, $\downarrow^m$, $\downarrow^{m+1}$, $\downarrow^{m+2}$, ..., $\downarrow^n$, but for each $\bar{Q}_i$ there is naturally only one possible direction for each $k \in \{1 , 2 , ... , n\}$. If one of the cubes is located in direction $\uparrow^k$ and another is in direction $\downarrow^k$, $k \leq m$, then the distance between these two cubes is at least $\frac{1}{3} l(Q_i)$. Thus we can choose one of these cubes into the sequence $(Q_{i_j}')_{j=1}^{\infty}$ the same way as in the first case and there are still infinitely many cubes $Q_l$, $l > i$, such that $3 Q_l \cap \bar{Q}_i = \varnothing$.

Let us assume that there are no two cubes $\bar{Q}_i$ that are located in opposite directions from $\frac{1}{3} Q_i$. Then by similar reasoning as in the proof of Proposition \ref{paatulos}, we know that at least one of the cubes $\bar{Q}_i$ must be located in direction $\downarrow^k$ for some $k > m$. Let us denote such a cube by $Q_{i,1}$. If there are infinitely many cubes $Q_l$, $l > i$, such that $3 Q_l \cap Q_{i,1} = \varnothing$, then we may choose $Q_{i,1}$ into the sequence $(Q_{i_j}')_{j=1}^{\infty}$.

Assume that there are only finitely many cubes $Q_l$, $l > i$, such that $3 Q_l \cap Q_{i,1} = \varnothing$. Then there are infinitely many cubes $Q_l$, $l > i$, such that $3 Q_l \cap Q_{i,1} \neq \varnothing$. Let $Q_l$ be one of these cubes. Because $l(Q_l) \leq \frac{2}{9} l(Q_i)$, we know that $P_k(3 Q_l) \subset P_k(X)$. This is because the distance from $P_k(Q_{i,1})$ to $\partial P_k(X)$ is at least $\frac{2}{3} l(Q_i)$.

In addition for every $1 \leq s \leq m$ we have $P_s(3 Q_l) \subset P_s(X)$. This is because of how the cube $\bar{Q}$ was chosen in Lemma \ref{tokalemma}. Because the cube $Q_{i,1}$ is located in direction $\downarrow^k$ from the cube $\frac{1}{3} Q_i$, we have either $Q_{i,1} \subset \tilde{Q}_i$ (where $\tilde{Q}_i$ is as in Lemma \ref{tokalemma}) or the projection of $\tilde{Q}_i$ and the projection of $Q_{i,1}$ coincide for every base vector except $v_k$.
Thus $P_s(Q_{i,1}) \subset P_s(\tilde{Q}_i)$.
Because $P_s(3 Q_i) \subset P_s(X)$ for every $1 \leq s \leq m$, we get that the distance from $P_s(\tilde{Q}_i)$ to $\partial P_s(X)$ is at least $\frac{2}{3} l(Q_i)$. See Figure \ref{vanishingsmallcubeskuva} to get a better understanding of the situation.

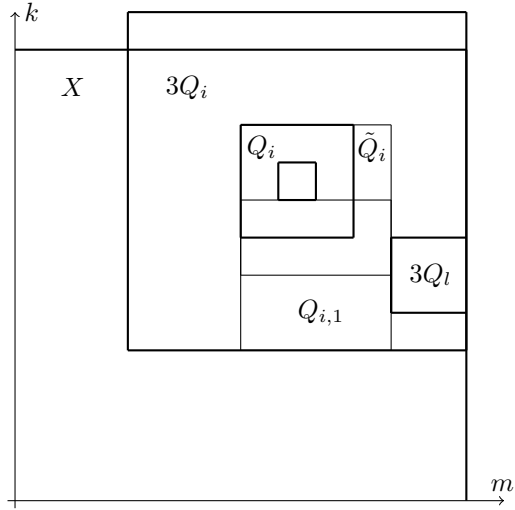
\begin{figure}[ht]
\centering
\begin{tikzpicture}

\draw[black,->] (-0.1,0) -- (6.5,0);
\draw[black,->] (0,-0.1) -- (0,6.5) ;

\draw (0,6.5) node[right] {$k$};
\draw (6.2,0.2) node[right] {$m$};

% X
\draw[black, thick] (0,6)--(6,6);
\draw[black, thick] (6,0)--(6,6);
\draw (1.05,5.5) node[left] {$X$};

% Q_i
\draw[black, thick] (3,3.5)--(3,5);
\draw[black, thick] (4.5,3.5)--(4.5,5);
\draw[black, thick] (3,3.5)--(4.5,3.5);
\draw[black, thick] (3,5)--(4.5,5);
\draw (3.6,4.7) node[left] {$Q_i$};

% Q_i/3
\draw[black, thick] (3.5,4)--(3.5,4.5);
\draw[black, thick] (3.5,4)--(4,4);
\draw[black, thick] (4,4)--(4,4.5);
\draw[black, thick] (3.5,4.5)--(4,4.5);
%\draw (4.2,4.2) node[left] {$\frac{1}{3} Q_i$};

% 3 Q_i
\draw[black, thick] (1.5,2)--(1.5,6.5);
\draw[black, thick] (6,2)--(6,6.5);
\draw[black, thick] (1.5,2)--(6,2);
\draw[black, thick] (1.5,6.5)--(6,6.5);
\draw (2.7,5.5) node[left] {$3 Q_i$};

% \tilde{Q}_i
\draw[black] (3,3)--(3,5);
\draw[black] (3,3)--(5,3);
\draw[black] (5,3)--(5,5);
\draw[black] (3,5)--(5,5);
\draw (5.07,4.7) node[left] {$\tilde{Q}_i$};

% Q_i,1
\draw[black] (3,2)--(3,4);
\draw[black] (3,2)--(5,2);
\draw[black] (5,2)--(5,4);
\draw[black] (3,4)--(5,4);
\draw (4.5,2.5) node[left] {$Q_{i,1}$};

% 3 Q_l
\draw[black, thick] (5,2.5)--(5,3.5);
\draw[black, thick] (5,2.5)--(6,2.5);
\draw[black, thick] (5,2.5)--(5,3.5);
\draw[black, thick] (5,3.5)--(6,3.5);
\draw (5.93,3) node[left] {$3 Q_l$};

\end{tikzpicture}

\caption{An example of how the cubes $X$, $Q_i$, $\tilde{Q}_i$, $Q_{i,1}$ and $Q_l$ may be situated with respect to each other. In the picture we have the projections of these cubes to the $mk$-plane. The cube $Q_{i,1}$ is located in direction $\downarrow^k$ from $\frac{1}{3} Q_i$. Therefore we have either $Q_{i,1}$ as shown in the picture or $Q_{i,1} \subset \tilde{Q}_i$.}

\label{vanishingsmallcubeskuva}

\end{figure}

In conclusion if $3 Q_i \not\subset X$ and there are $m$ base vectors $v_s$ such that $P_s(3 Q_i) \subset P_s(X)$, then we can find a cube $Q_{i,1} \subset X \cap 3 Q_i$ such that \eqref{kolmasehto} holds and there are infinitely many cubes $Q_l$, $l > i$, such that $Q_{i,1} \cap 3 Q_l = \varnothing$, or, if the first option is not possible, there exist infinitely many cubes $Q_l$, $l > i$, such that $P_s(3 Q_l) \subset P_s(X)$ for at least $m+1$ base vectors $v_s$.
Because there are only $n$ base vectors in total, this latter option can only happen at most $n$ times. Thus we can always find infinitely many pairwise disjoint cubes $(Q_{i_j}')_{j=1}^{\infty}$ in $X$ that all satisfy \eqref{kolmasehto}.

In both cases we get
\begin{equation*}
\sum_{j=1}^{\infty} |Q_{i_j}'| \left( \fint_{Q_{i_j}'} \left| f - f_{Q_{i_j}'} \right| \right)^p
\geq \sum_{j=1}^{\infty} c(n,p) A^p
= \infty.
\end{equation*}
This contradicts with the fact that $f \in JN_p(X)$. Therefore we conclude that $A = 0$.

Now assume only that $f \in JN_p(X)$. Then there is a constant $b$ such that $f - b \in L^{p,\infty}(X)$. Also from \eqref{itseisarvoepayhtalo} we get that $|f - b| \in JN_p(X) \cap L^{p,\infty}(X)$. Then using the same argument as earlier we get
\begin{equation*}
\lim_{a \to 0} \sup_{\substack{ Q \subseteq X \\ l(Q) \leq a }} |Q|^{1/p} \fint_Q |f|
\leq \lim_{a \to 0} \sup_{\substack{ Q \subseteq X \\ l(Q) \leq a }} |Q|^{1/p} \fint_Q |f - b| + \lim_{a \to 0} \sup_{\substack{ Q \subseteq X \\ l(Q) \leq a }} |Q|^{1/p} \fint_Q |b|
= 0.
\end{equation*}
This completes the proof.
\end{proof}

%\textbf{Statements and declarations}
%The research leading to these results received funding from […] under Grant Agreement No[…].
%\\
%tai vaihtoehtoisesti
%\\
%The authors have no relevant financial or non-financial interests to disclose.

%An interesting topic of further study is whether there exists a more quantitative version of Theorems \ref{vanishingbigcubes} and \ref{vanishingsmallcubes}. For example we might have
%\begin{equation*}
%\sup_{\substack{Q \subseteq X \\ l(Q) \leq a}} |Q|^{\frac{1}{p}} \fint_Q |f|
%\leq g(n,p,f,a),
%\end{equation*}
%where $g$ is some function where it would be easy to see that $\lim_{a \to 0} g = 0$.

\end{document}